\documentclass[reqno]{amsart}
\pdfoutput1

\usepackage{amssymb,color,hyperref,mathrsfs,stmaryrd}
\usepackage{amsmath}
\usepackage{mathtools}

\usepackage[usenames,dvipsnames]{xcolor}
\usepackage[curve,matrix,arrow]{xy}

\usepackage{amsmath,amsfonts,amsthm,amssymb,amscd, verbatim, graphicx}

\usepackage{pinlabel}
\usepackage{lscape}
\usepackage{tabularx}
\usepackage{latexsym}
\usepackage{hyperref}
\newtheorem*{thma}{Theorem A}
\newtheorem*{thmb}{Theorem B}

\newtheorem*{lemma}{Lemma~1}

\theoremstyle{definition}

\theoremstyle{remark}

\numberwithin{equation}{section}

\newcommand{\Aut}{\operatorname{Aut}}

\newcommand{\Id}{\operatorname{Id}}

\newcommand{\Syl}{\operatorname{Syl}}

\newcommand{\Inn}{\operatorname{Inn}}

\renewcommand{\Gamma}{\varGamma}
\renewcommand{\epsilon}{\varepsilon}

\renewcommand{\leq}{\leqslant}

\newcommand{\F}{\mathcal{F}}

\newcommand{\E}{\mathcal{E}}

\renewcommand{\phi}{\varphi}

\begin{document}

%%
%% The title of the paper goes here.  Edit to your title.
%%

\title{Centralizers of normal subgroups and the $Z^*$-theorem}
 
%%
%% Now edit the following to give your name and address:
%% 

\author{E. Henke}
\address{Institute of Mathematics, University of Aberdeen, U.K.}
\email{ellen.henke@abdn.ac.uk}
\author{J. Semeraro}
\address{Heilbronn Institute for Mathematical Research, Department of Mathematics, University of Bristol, U.K.}
\email{js13525@bristol.ac.uk}

%%
%% If there are three of more authors they are added in the obvious
%% way. 
%%

%%%
%%% The following is for the abstract.  The abstract is optional and
%%% if not used just delete, or comment out, the following.
%%%

\begin{abstract}
Glauberman's $Z^*$-theorem and analogous statements for odd primes show that, for any prime $p$ and any finite group $G$ with Sylow $p$-subgroup $S$, the centre of $G/O_{p^\prime}(G)$ is determined by the fusion system $\F_S(G)$. Building on these results we show a statement that seems a priori more general: For any normal subgroup $H$ of $G$ with $O_{p^\prime}(H)=1$, the centralizer $C_S(H)$ is expressed in terms of the fusion system $\F_S(H)$ and its normal subsystem induced by $H$.
\end{abstract}

\maketitle


\bigskip

%Glauberman's $Z^*$-theorem states a relationship between the global structure of a finite group and the fusion of $2$-elements. Analogous results for arbitrary primes have been shown using the classification of finite simple groups. We give a precise statement here in terms of fusion systems.

Throughout $p$ is a prime. Glauberman's $Z^*$-theorem \cite{Glauberman:1966} and its generalization to odd primes, which is shown using the classification of finite simple groups (see \cite{Xiao:1991} and \cite{Guralnick/Robinson:1993}), can be reformulated as follows:

\begin{thma}
 Let $G$ be a finite group with $O_{p^\prime}(G)=1$, and $S\in\Syl_p(G)$. Then $Z(G)=Z(\F_S(G))$. 
\end{thma}

We refer the reader here to \cite{Aschbacher/Kessar/Oliver:2011a} for basic definitions and results regarding fusion systems; see in particular Definitions~I.4.1 and I.4.3 for the definition of central subgroups and the centre $Z(\F)$. A more common formulation of the $Z^*$-theorem states that, assuming the hypothesis of Theorem~A, we have $t\in Z(G)$ if and only if $t^G\cap S=\{t\}$ for every element $t\in S$ of order $p$. Given a normal subgroup $H$ of a finite group $G$, a Sylow $p$-subgroup $S\in\Syl_p(G)$, and an element $t\in S$ of order $p$, one can apply the $Z^*$-theorem with $H\langle t\rangle$ in place of $G$ to obtain the following corollary: Provided $O_{p^\prime}(H)=1$, we have $t^H\cap S=\{t\}$ if and only $t\in C_S(H)$. In this short note, we use Theorem~A to give a less obvious characterization of $C_S(H)$.

\smallskip

Given a saturated fusion system $\F$ on a finite $p$-group $S$ and a normal subsystem $\E$ of $\F$ on $T\leq S$,  Aschbacher \cite[(6.7)(1)]{Aschbacher:2011} showed that the set of subgroups $X$ of $C_S(T)$ with $\E\subseteq C_\F(X)$ has a largest member $C_S(\E)$. He furthermore constructed a normal subsystem $C_\F(\E)$ on $C_S(\E)$, the centralizer of $\E$ in $\F$; see \cite[Chapter~6]{Aschbacher:2011}. Note that $C_S(\E)$ depends not only on $S$ and $\E$ but also on the fusion system $\F$ in which both $S$ and $\E$ are contained. 

\smallskip

The definition of $C_S(\E)$ generalizes the definition of $Z(\F)$ since $C_S(\F)=Z(\F)$. Moreover, for every normal subgroup $H$ of a finite group $G$ with Sylow $p$-subgroup $S$, $\F_{S\cap H}(H)$ is a normal subsystem of $\F_S(G)$ by \cite[I.6.2]{Aschbacher/Kessar/Oliver:2011a}. Thus, the following theorem, which we prove later on, can be seen as a generalization of Theorem~A.

\begin{thmb}
Let $G$ be a finite group and let $S$ be a Sylow $p$-subgroup of $G$. Let $H \unlhd G$ with $O_{p^\prime}(H)=1$. Then $C_S(\F_{S\cap H}(H))=C_S(H)$.
\end{thmb}

In the statement of Theorem~B it is understood that $C_S(\F_{S\cap H}(H))$ is formed inside of $\F_S(G)$. The result says in other words that, under the hypothesis of Theorem~B, for any $X\leq S$ with $\F_{S\cap H}(H)\subseteq C_{\F_S(G)}(X)$, we have $X\leq C_S(H)$. This is not true if one drops the assumption that $H$ is normal in $G$ as the following example shows: Let $G:=G_1\times G_2$ with $G_1\cong G_2\cong S_3$. Set $p=3$, $S=O_3(G)$, $S_i:=O_3(G_i)$ and let $R$ be a subgroup of $G$ of order $2$ which acts fixed point freely on $S$. Set $H:=S_1\rtimes R$. Then $S_1=S\cap H\in\Syl_3(H)$ and $\F_{S_1}(H)=\F_{S_1}(G_1)\subseteq C_{\F_S(G)}(S_2)$ as $S_2=C_S(G_1)$. However, $S_2\not\leq C_S(H)$ by the choice of $R$.   

\smallskip

Theorem~B was conjectured by the second author of this paper in \cite{Semeraro:2014}. Our proof of Theorem~B builds on Theorem~A and the reduction uses only elementary group theoretical results. Essential is the following lemma, whose proof is self-contained apart from using the conjugacy of Hall-subgroups in solvable groups.

\begin{lemma}
Let $G$ be a finite group with Sylow $p$-subgroup $S$ and a normal subgroup $H$. Let $P\leq S$ such that $P\cap H$ is centric in $\F_{S\cap H}(H)$. Then for every $p^\prime$-element $\phi\in\Aut_G(P)$ with $[P,\phi]\leq P\cap H$ and $\phi|_{P\cap H}\in\Aut_H(P\cap H)$, we have $\phi\in\Aut_H(P)$.
%$$O^p(\Aut_H(P))=\langle\phi\in\Aut_G(P) \mid \phi\mbox{ a $p^\prime$-element},\;[P,\phi]\leq P\cap H\;\phi|_{P\cap H}\in\Aut_H(P)\rangle.$$ 
\end{lemma}

\begin{proof}
 This is \cite[Proposition~3.1]{Henke:2013}.
\end{proof}

\begin{proof}[Proof of Theorem~B]
We assume the hypothesis of Theorem~B. Furthermore, we set $\F:=\F_S(G)$, $T:=S\cap H$ and $\E:=\F_T(H)$. If a homomorphism $\phi$ between subgroups $A$ and $B$ of $T$ is induced by conjugation with an element $h\in H$, then $\phi$ extends to $c_h:AC_S(H)\rightarrow BC_S(H)$ and $c_h$ restricts to the identity on $C_S(H)$. Thus $\E \subseteq C_\F(C_S(H))$, so by the definition of $C_S(\E)$, we have $C_S(H) \leq C_S(\E)$. To prove the converse inclusion, choose $t \in C_S(\E)$. Define: $$G_0:=H\langle t \rangle \hspace{5mm} \mbox{ and } \hspace{5mm} S_0:=T\langle t \rangle,$$ so that plainly $S_0$ is a Sylow $p$-subgroup of $G_0$ and $\F_0:=\F_{S_0}(G_0)$ is a saturated fusion system on $S_0$. Note also that $O_{p^\prime}(G_0)=1$ as $O^p(G_0)=O^p(H)$ and $O_{p^\prime}(H)=1$ by assumption.

\smallskip

By Theorem~A, $Z(\F_0)=Z(G_0) \leq C_S(H)$. It thus suffices to prove $t\in Z(\F_0)$. As $t\in C_S(\E)\leq C_S(T)$, $t\in Z(S_0)$. Let $P$ be a subgroup of $S_0$ which is centric radical and fully normalized in $\F_0$. Then $t\in Z(S_0)\leq C_{S_0}(P)\leq P$. It is sufficient to prove $[t,\Aut_{\F_0}(P)]=1$. For as $P$ is arbitrary, Alperin's fusion theorem \cite[Theorem~3.6]{Aschbacher/Kessar/Oliver:2011a} implies then $t\in Z(\F_0)$. As $P$ is fully $\F_0$-normalized, $\Aut_{S_0}(P)\in\Syl_p(\Aut_{\F_0}(P))$ and thus $\Aut_{\F_0}(P)=\Aut_{S_0}(P)O^p(\Aut_{\F_0}(P))$. Note that $[t,\Aut_{S_0}(P)]=1$ as $t\in Z(S_0)$. Hence, it is enough to prove 
$$[t,O^p(\Aut_{\F_0}(P))]=1.$$ 
Let $\phi\in\Aut_{\F_0}(P)$ be a $p^\prime$-element. Since $O^p(H)=O^p(G_0)$, we have $O^p(\Aut_{\F_0}(P))=O^p(\Aut_H(P))$. In particular, $\phi\in\Aut_H(P)$ and thus $\phi|_{P\cap T}\in\Aut_H(P\cap T)=\Aut_{\E}(P\cap T)$. As $t\in P\leq S_0=T\langle t\rangle$, we have $P=(P\cap T) \langle  t \rangle$. Moreover, $t\in C_S(\E)$ implies that $\E \subseteq C_\F(\langle t\rangle)$. Hence, $\varphi|_{P\cap T}$ extends to $\psi \in \Aut_\F(P)$ with the property that $t\psi=t$. Note that $o(\psi)=o(\phi|_{P\cap T})$ and thus $\psi$ is a $p^\prime$-element as $\phi$ has order prime to $p$. Moreover, plainly $[P,\psi]\leq P\cap T$ and $\psi|_{P\cap T}=\phi|_{P\cap T}\in\Aut_H(P\cap T)$. Since $\E \unlhd \F_0$, $P \cap T$ is $\E$-centric by \cite[7.18]{Aschbacher:2011}. Now it follows from Lemma~1 that $\psi\in \Aut_H(P)$. Thus, $\chi:=\varphi \circ \psi^{-1} \in \Aut_H(P)\leq \Aut_{\F_0}(P)$. Clearly $\chi|_{P \cap T} = \Id$ as $\psi$ extends $\phi|_{P\cap T}$. Moreover, using that $H$ is normal in $G$, we obtain $[P,\chi] \leq [P,\Aut_H(P)]=[P,N_H(P)]\leq P\cap H=P \cap T$.  Hence, by \cite[Lemma~A.2]{Aschbacher/Kessar/Oliver:2011a}, $\chi\in C_{\Aut_{\F_0}(P)}(P/(P\cap T))\cap C_{\Aut_{\F_0}(P)}(P\cap T)=O_p(\Aut_{\F_0}(P))=\Inn(P)$ as $P$ is radical in $\F_0$. As $\Inn(P)\leq \Aut_{S_0}(P)$ and $[t,\Aut_{S_0}(P)]=1$, it follows $t\chi=t$. By the choice of $\psi$, also $t\psi=t$ and consequently $t\phi=t$. Since $\phi$ was chosen to be an arbitrary $p^\prime$-element in $\Aut_{\F_0}(P)$ and $O^p(\Aut_{\F_0}(P))$ is the subgroup generated by these elements, it follows that $[t,O^p(\Aut_{\F_0}(P))]=1$. As argued above, this yields the assertion.
\end{proof}

\bibliographystyle{amsplain}
\bibliography{repcoh}

\end{document}